\documentclass[a4paper, 10pt]{amsart}

\usepackage{amsmath,amssymb, hyperref, color}

\newtheorem{theorem}{Theorem}[section]
\newtheorem{lemma}[theorem]{Lemma}

\newtheorem{proposition}[theorem]{Proposition}

\theoremstyle{definition}

\newtheorem{examples}[theorem]{Examples}

\newtheorem{remarks}[theorem]{\bf Remarks}

\newcommand{\N}{\mathbb N}
\newcommand{\Z}{\mathbb Z}

\newcommand{\Q}{\mathbb Q}

\newcommand{\red}{{\text{\rm red}}}

\newcommand{\BF}{\text{\rm BF}}

 \DeclareMathOperator{\ord}{ord}

 \DeclareMathOperator{\supp}{supp}

 \DeclareMathOperator{\Ca}{\it Ca}

\renewcommand{\t}{\, | \,}
\renewcommand{\time}{\negthinspace \times \negthinspace}

\numberwithin{equation}{section}

\begin{document}

\title{Minimal relations and catenary degrees in Krull monoids}

\address{University of Graz, NAWI Graz \\ Institute for Mathematics and Scientific Computing \\ Heinrichstra{\ss}e 36\\ 8010 Graz, Austria}

\email{alfred.geroldinger@uni-graz.at}

\address{Mathematical College, China University of Geosciences (Beijing), Haidian District, Beijing, China}

\email{fys@cugb.edu.cn}

\author{Yushuang Fan and Alfred Geroldinger}

\thanks{This work was supported by the Austrian Science Fund FWF (Project Number P P26036-N26), by the  NSFC (Grant Number 11401542), and by the CSC}



\begin{abstract}
Let $H$ be a Krull monoid with class group $G$. Then $H$ is factorial if and only if $G$ is trivial. Sets of lengths and sets of catenary degrees are well studied invariants describing the arithmetic of $H$ in the non-factorial case. In this note we focus on the set $\Ca (H)$ of catenary degrees of $H$ and on the set $\mathcal R (H)$ of distances in minimal relations. We show that every finite nonempty subset of $\N_{\ge 2}$ can be realized as the set of catenary degrees of a Krull monoid with finite class group. This answers Problem 4.1 of \cite{N-P-T-W16a}. Suppose in addition that every class of $G$ contains a prime divisor. Then $\Ca (H)\subset \mathcal R (H)$ and $\mathcal R (H)$ contains a long interval. Under a reasonable condition on the Davenport constant of $G$, $\mathcal R (H)$ coincides with this interval and the maximum equals the catenary degree of $H$.
\end{abstract}

\maketitle

\medskip
\section{Introduction} \label{1}
\medskip

In this note we study the arithmetic of atomic monoids with a focus on Krull monoids. This setting includes Krull domains and hence all integrally closed noetherian domains.
By an atomic monoid, we mean a commutative cancelative semigroup with unit element with the property that every non-unit can be written as a finite product of atoms (irreducible elements). Let $H$ be an atomic monoid. Then $H$ is factorial if and only if every equation $u_1 \cdot \ldots \cdot u_k = v_1 \cdot \ldots \cdot v_{\ell}$, with $k, \ell \in \N$ and atoms $u_1, \ldots , u_k, v_1, \ldots, v_{\ell}$, implies that $k=\ell $ and that, after  renumbering if necessary, $u_i$ and $v_i$ just differ by a unit for all $i \in [1,k]$. It is well-known that $H$ is factorial if and only if it is a Krull monoid with trivial class group.

Suppose that $H$ is atomic but not factorial. Then there is an element $a \in H$ having two distinct factorizations $z$ and $z'$, which can be written in the form
\[
z = u_1 \cdot \ldots \cdot u_kv_1 \cdot \ldots \cdot v_{\ell} \quad \text{and} \quad
z'= u_1  \cdot \ldots \cdot u_k w_1 \cdot \ldots \cdot w_{m}
\]
where $k \in \N_0$, $\ell, m \in \N$, all $u_r, v_s, w_t$ are atoms, and no $v_s$ is associated to any $w_t$ with $r \in [1,k]$, $s \in [1, \ell]$, and $t \in [1,m]$. Then we call $\mathsf d (z,z')=\max \{\ell, m\}$ the distance between the factorizations $z$ and $z'$. The catenary degree $\mathsf c (a)$ of $a$ is the smallest $N \in \N_0 \cup \{\infty\}$ such that for any two factorizations $y,y'$ of $a$ there are factorizations $y=y_0, y_1, \ldots, y_n=y'$ of $a$ such that the distance $\mathsf d (y_{i-1}, y_i)$, where $i \in [1,n]$, of each two subsequent factorizations is bounded by $N$. The catenary degree $\mathsf c (H)$ of $H$ is defined as the supremum of the catenary degrees $\mathsf c (a)$ over all elements $a \in H$.

In this note we study the set of catenary degrees and the set $\mathcal R (H)$ of distances in minimal relations of $H$. More precisely, $\mathcal R (H)$ is defined as the set of all $d \in \N$ having the following property:
\begin{itemize}
\item[] There is an element $a \in H$ having two distinct factorizations $z$ and $z'$ with distance $\mathsf d (z,z')=d$ but there are no factorizations $z=z_0, z_1, \ldots, z_n=z'$ of $a$ such that $\mathsf d (z_{i-1}, z_i) < d$ for all $i \in [1,n]$.
\end{itemize}
The key idea of the present note is to study $\mathcal R (H)$ with the help of a crucial subset $\daleth^* (H)$, defined as
\[
\daleth^* (H) = \{\min( \mathsf L (uv) \setminus \{2\}) \mid u, v \ \text{are atoms and the set of lengths $\mathsf L (uv)$ has at least two elements} \} \,.
\]
Thus we consider four closely related sets of invariants, namely $\mathcal R (H)$, $\Ca (H)$, $\daleth^* (H)$, together with the well-studied set of distances $\Delta (H)$ (also called the delta set of $H$). These sets satisfy some straightforward inclusions (see Equation \eqref{inclusion} and Lemma \ref{2.1}) but they are different in general (Examples \ref{2.3}). First we show that every finite nonempty subset $C \subset \N_{\ge 2}$ can be realized as the set of catenary degrees of a Krull monoid with finite class group (Proposition \ref{3.2}). This answers Problem 4.1 of \cite{N-P-T-W16a}.

In contrast to the wildness provided by this realization result, each of the sets $\mathcal R (H)$, $\Ca (H)$, $\daleth^* (H)$, and $\Delta (H)$  is very  structured if $H$ is a  Krull monoid such that each class contains a prime divisor. This assumption holds true, among others, for holomorphy rings in global fields and for all semigroup rings which are Krull.
Indeed, under this assumption on the prime divisors the set $\daleth^* (H)$ is an interval (Proposition \ref{3.3}). Under a reasonable condition on the Davenport constant of the class group, the sets $\daleth^* (H) \cup \{2\}$ and $\mathcal R (H)$ coincide, they are  intervals, and their maxima equal the catenary degree. In order to formulate our main result, we recall that for an abelian group $G$, $\mathsf D (G)$ denotes the Davenport constant of $G$ (the assumption $\mathsf D (G)=\mathsf D^* (G)$ will be analyzed  in Remark \ref{3.5}).

\medskip
\begin{theorem} \label{1.1}
Let $H$ be a Krull monoid with  class group $G$ such that every class contains a prime divisor. Then $\daleth^* (H)$ is an interval with $\daleth^* (H) \subset \mathcal R (H)$. Moreover, if \ $\mathsf D (G) = \mathsf D^* (G) \in \N_{\ge 4}$, then
      \begin{equation} \label{mainresult}
      \mathcal R (H) = \daleth^* (H) \cup \{2\} = \big(2+\Delta (H)\big) \cup \{2\} = [2, \mathsf c (H)] \,.
      \end{equation}
In detail, we have
\begin{enumerate}
\item If \ $|G|=1$, then $\mathcal R (H)=\emptyset$, and if $|G|=2$, then $\mathcal R (H)=\{2\}$. If $G$ is infinite, then $\mathcal R (H)=\N_{\ge 2}$.

\smallskip
\item  If \ $\mathsf D (G)=3$ and every nonzero class contains precisely one prime divisor, then $\mathcal R (H) = \daleth^* (H) = \{3\}$.

\smallskip
\item If either \ $\big( \mathsf D (G)=3$ and there is a nonzero class containing at least two distinct prime divisors$\big)$ or if \ $\mathsf D (G) \in \N_{\ge 4}$, then
    $\min \mathcal R (H) = 2$ and  $\min \daleth^* (H) = 3$.
\end{enumerate}
\end{theorem}

\medskip
\section{Background on factorizations and Krull monoids} \label{2}
\medskip

We denote by $\N$ the set of positive integers and we put $\N_0 = \N \cup \{0\}$. For integers $a, b \in \Z$, we denote by $[a,b]=\{c \in \Z \mid a \le c \le b\}$ the discrete interval. For subsets $A, B \subset \Z$, $A+B = \{a+b\mid a \in A, b \in B\}$ denotes the sumset, $-A=\{-a \mid a \in A\}$, and $y+A = \{y\}+A$ for every $y \in \Z$. If $A = \{a_1, \ldots, a_k\}$ with $k \in \N_0$ and $a_1 < \ldots < a_k$, then $\Delta (A)= \{ a_{\nu+1} - a_{\nu} \mid \nu \in [1, k-1] \}$ is the set of distances of $A$. Thus $|A|\le 1$ if and only if $\Delta (A)=\emptyset$. For $A \subset \N$, we denote by $\rho (A) = \sup A/\min A \in \Q_{\ge 1} \cup \{\infty\}$ the elasticity of $A$, and we set $\rho (\{0\})=1$.
Let $G$ be an additive abelian group and $r \in \N$. An $r$-tuple $(e_1, \ldots, e_r)$ of elements from $G$ is said to be independent if each $e_i$ is nonzero and $\langle e_1, \ldots, e_r\rangle = \langle e_1 \rangle \oplus \ldots \oplus \langle e_r \rangle$.

By a monoid, we mean a commutative cancelative semigroup with unit element, and we will use multiplicative notation. Let $H$ be monoid. We denote by $H^{\times}$ the group of units, by $H_{\red}=H/H^{\times}$ the associated reduced monoid, by $\mathcal A (H)$ the set of atoms of $H$, and by $\mathsf q (H)$ the quotient group of $H$. For a set $P$, we denote by $\mathcal F (P)$ the free abelian monoid with basis $P$. An element $a \in \mathcal F (P)$ will be written in the form
\[
a = \prod_{p \in P} p^{\mathsf v_p (a)} \quad \text{with} \quad \mathsf v_p (a) \in \N_0 \ \text{and} \ \mathsf v_p (a)=0 \ \text{for almost all} \ p \in P \,,
\]
and we call
\[
|a|_F = |a| = \sum_{p \in P}\mathsf v_p (a) \in \N_0 \quad \text{the length of} \ a \,.
\]

\smallskip
\noindent
{\bf Factorizations and sets of lengths.}
The monoid $\mathsf Z (H) = \mathcal F ( \mathcal A (H_{\red}))$ is the  factorization monoid of $H$, and the factorization homomorphism $\pi \colon \mathsf Z (H) \to H_{\red}$
maps a factorization onto the element it factors. For an element $a \in H$, we call
\begin{itemize}
\item $\mathsf Z_H (a) = \mathsf Z (a) = \pi^{-1}(aH^{\times}) \subset \mathsf Z (H) \ \text{the set of factorizations of} \ a$, \ and

\item $\mathsf L_H (a) = \mathsf L (a) = \{|z| \mid z \in \mathsf Z (a) \} \subset \N_0 $ \  the set of lengths of $a$.
\end{itemize}
Note that $\mathsf L (a)=\{0\}$ if and only if $a \in H^{\times}$, and that $1 \in \mathsf L (a)$ if and only if $a \in \mathcal A (H)$, and then $\mathsf L (a)=\{1\}$. The monoid is said to be
\begin{itemize}
\item {\it atomic} if \ $\mathsf Z (a)\ne \emptyset$ for all $a \in H$ (equivalently, every nonunit is a finite product of atoms),

\item {\it factorial} if \ $|\mathsf Z (a)|=1$ for all $a \in H$ (equivalently, $H_{\red}$ is free abelian),

\item {\it half-factorial} if \ $|\mathsf L (a)|=1$    for all $a \in H$,

\item a {\it \BF-monoid} if $\mathsf L (a)$ is finite and nonempty  for all $a \in H$.
\end{itemize}
From now on we suppose that $H$ is a BF-monoid (which holds true if $H$ is $v$-noetherian). We denote by  $\mathcal L (H) = \{\mathsf L (a) \mid a \in H \}$  the {\it system of sets of lengths} of $H$, by
\[
\Delta (H) = \bigcup_{L \in \mathcal L (H)} \Delta (L) \ \subset \N
\]
the {\it set of distances} of $H$ (also called the {\it delta set} of $H$), and if $H$ is not half-factorial, then it follows from \cite[Proposition 1.4.4]{Ge-HK06a} that
\begin{equation} \label{minequalsgcd}
\min \Delta (H) = \gcd \Delta (H) \,.
\end{equation}
We study sets of distances and sets of catenary degrees via the following set $\daleth^* (H)$, which is defined as
\[
\begin{aligned}
\daleth^* (H)&  = \{ \min (\mathsf L (uv) \setminus \{2\}) \mid u,v \in \mathcal A (H), |\mathsf L (uv)|>1\}  \\
 & = \{ \min (L  \setminus \{2\}) \mid 2 \in L \in \mathcal L (H), |L|>1 \} \subset \N_{\ge 3} \,.
\end{aligned}
\]
By definition, we have
\begin{equation} \label{inclusion}
\daleth^* (H) \subset 2+\Delta (H) \subset \N_{\ge 3},
\end{equation}
$H$ is half-factorial if and only if $\Delta (H)=\emptyset$, and if this holds then $\daleth^* (H)=\emptyset$.

\smallskip
\noindent
{\bf Catenary degrees and minimal relations.} Any two factorizations $z, z' \in \mathsf Z (H)$ can be written as
\[
z = u_1 \cdot \ldots \cdot u_k v_1 \cdot \ldots \cdot v_{\ell} \quad \text{and} \quad z'= u_1 \cdot \ldots \cdot u_k w_1 \cdot \ldots \cdot w_m \,,
\]
where $k, \ell, m \in \N_0$, $u_1, \ldots, u_k,v_1, \ldots, v_{\ell},w_1, \ldots, w_m \in \mathcal A (H_{\red})$, and $\{v_1, \ldots, v_{\ell} \} \cap \{w_1, \ldots, w_m\}=\emptyset$, and then $\mathsf d (z,z')=\max \{\ell,m\}$ is the {\it distance} between $z$ and $z'$. Thus $z=z'$ if and only if $\mathsf d (z,z')=0$. If $z, z' \in \mathsf Z (a)$ for some $a \in H$ and $N \in \N_0$, then a finite sequence $z_0, z_1, \ldots, z_k$ in $\mathsf Z (a)$ is called an $N$-chain of factorizations from $z$ to $z'$ if $z=z_0$, $z'=z_k$, and $\mathsf d (z_{i-1},z_i)\le N$ for each $i \in [1,k]$. The {\it catenary degree} $\mathsf c (a) \in \N_0 \cup \{\infty\}$ of $a$ is defined as the smallest $N \in \N_0 \cup \{\infty\}$ such that any two factorizations of $a$ can be concatenated by an $N$-chain, and
\[
\mathsf c (H) = \sup \{ \mathsf c (a) \mid a \in H \} \in \N_0 \cup \{\infty\}
\]
denotes the {\it catenary degree} of $H$. Clearly, $\mathsf c (a)=0$ if and only if $|\mathsf Z(a)|=1$, and hence $\mathsf c (H)=0$ if and only if $H$ is factorial. If $z,z' \in \mathsf Z (a)$ are distinct, then $2+ \big| |z|-|z'| \big| \le \mathsf d (z,z')$, and $|\mathsf Z (a)|>1$ implies that $2+\sup \Delta ( \mathsf L (a)) \le \mathsf c (a)$ (\cite[Lemma 1.6.2]{Ge-HK06a}). Thus, if $H$ is not factorial, then (\cite[Theorem 1.6.3]{Ge-HK06a}
\begin{equation} \label{inclusion2}
2 + \sup \Delta (H) \le  \mathsf c (H) \,.
\end{equation}
Since $H$ is a BF-monoid and $\mathsf c (a) \le \sup \mathsf L (a) < \infty$, the catenary degrees of all elements are finite, and we denote by
\[
\Ca (H) = \{\mathsf c (a) \mid a \in H \ \text{with} \ |\mathsf Z (a)|> 1 \} \subset \N_{\ge 2}
\]
the {\it set of catenary degrees} of all elements having at least two distinct factorizations. If $a \in \mathcal A (H)$, then $|\mathsf Z (a)|=1$, and if $|\mathsf Z (a)|=1$, then $\mathsf c (a)=0$. In order to simplify the statements of our results, we  define $\Ca (H)$ as the set of positive catenary degrees.
Furthermore, let $\mathcal R (H)$  be the set of all $d \in \N_{\ge 2}$ with the following property:
\begin{itemize}
\item[] There are an element $a \in H$ and two distinct factorizations $z, z' \in \mathsf Z (a)$ with $\mathsf d (z,z') = d$ such that  there is no $(d-1)$-chain of factorizations concatenating $z$ and $z'$.
\end{itemize}
Our first lemma gathers some elementary properties of all these concepts.

\medskip
\begin{lemma} \label{2.1}
Let $H$ be a \BF-monoid.
\begin{enumerate}
\item $H$ is factorial if and only if $\mathsf c (H)=0$ if and only if $\Ca (H)=\emptyset$ if and only if $\mathcal R (H)=\emptyset$.

\smallskip
\item $\Ca (H) \subset \mathcal R (H)$. If $H$ is not factorial, then  $2 \le \min \Ca (H)$ and $\mathsf c (H) = \sup \Ca (H) = \sup \mathcal R (H)$.

\smallskip
\item  $\daleth^* (H) \subset \mathcal R (H) \subset \N_{\ge 2}$.
\end{enumerate}
\end{lemma}

\begin{proof}
1. Since $H$ is factorial if and only if $|\mathsf Z (a)|=1$ for all $a \in H$, the assertion follows.

\smallskip
2. Let $a \in H$ having at least two distinct factorizations. By definition of $\mathsf c (a)$, there are factorizations $z, z' \in \mathsf Z (a)$ with $\mathsf d (z,z')=\mathsf c (a)$ which cannot be concatenated by a $(\mathsf c (a)-1)$-chain of factorizations. Thus $\mathsf c (a) \in \mathcal R (H)$. Suppose that $H$ is not factorial. Then there is an element $a \in H$ with $|\mathsf Z (a)|>1$, and for every such element we have $\mathsf c (a) \ge 2$  whence $\min \Ca (H)  \ge 2$. Since $\mathsf c (H) = \sup \{\mathsf c (a) \mid a \in H \ \text{with} \ |\mathsf Z (a)| > 1 \}$, it follows that $\mathsf c (H) = \sup \mathcal R (H)$.

\smallskip
3. Suppose that $\daleth^* (H) \ne \emptyset$, and let $a \in H$ with $2 \in \mathsf L (a)$ and $|\mathsf L (a)|>1$. Then there are $u_1,u_2, v_1, \ldots , v_{\ell} \in \mathcal A (H)$ with $u_1u_2 = v_1 \cdot \ldots \cdot v_{\ell}$ where $\ell = \min ( \mathsf L (a) \setminus \{2\})$. Since for any distinct two factorizations $x, x' \in \mathsf Z (a)$ we have $\mathsf d (x,x') \ge 2 + | |x| - |x'| |$  and $\ell = \min ( \mathsf L (a) \setminus \{2\})$, the factorizations $z=u_1u_2$ and $z'= v_1 \cdot \ldots \cdot v_{\ell}$ cannot be concatenated by an $(\ell -1)$-chain of factorizations. Thus $\ell = \mathsf d (z,z') \in \mathcal R (H)$.
\end{proof}

\noindent
{\bf Krull monoids and zero-sum sequences.} A monoid homomorphism $\varphi \colon H \to D$ is called
\begin{itemize}
\item a {\it divisor homomorphism} if $\varphi (a) \t \varphi (b)$ implies that $a \t b$ for all $a, b \in H$,

\item a {\it divisor theory} (for $H$) if $\varphi$ is a divisor homomorphism, $D$ is free abelian, and for every $\alpha \in D$ there are $a_1, \ldots, a_m \in H$ such that $\alpha = \gcd ( \varphi (a_1), \ldots, \varphi (a_m) )$.
\end{itemize}
A monoid $H$ is a {\it Krull monoid} if it satisfies one of the following equivalent conditions (\cite[Theorem 2.4.8]{Ge-HK06a}){\rm \,:}
\begin{enumerate}
\item[(a)] $H$ is completely integrally closed and satisfies the ACC on divisorial ideals.
\item[(b)] $H$ has a divisor theory.
\item[(c)] There is a divisor homomorphism from $H$ to a factorial monoid.
\end{enumerate}
Suppose that $H$ is a Krull monoid. Then there is a free abelian monoid $F = \mathcal F (P)$ such that the embedding $H_{\red} \hookrightarrow F$ is a divisor theory. The group $\mathcal C (H) = \mathsf q (F)/\mathsf q (H_{\red})$ is the (divisor) class group of $H$ and $G_P = \{ [p] = p \mathsf q (H_{\red}) \mid p \in P \} \subset \mathcal C (H)$ is the set of classes containing prime divisors. We refer to \cite{HK98} and \cite{Ge-HK06a} for detailed presentations of the theory of Krull monoids. Here we just recall that  a domain $R$ is a Krull domain if and only if its monoid of nonzero elements is a Krull monoid, and Property (a) shows that every integrally closed noetherian domain is Krull. Holomorphy rings in global fields and regular congruence monoids in these domains are Krull monoids with finite class group and every class contains a prime divisor (\cite{Ge-HK06a}). Furthermore, semigroup rings which are Krull have the property that every class contains a prime divisor (\cite{Ch11a}).

We continue with a Krull monoid having a combinatorial flavor whose significance will become obvious in Lemma \ref{2.2}. Let $G$ be an additive abelian group and $G_0 \subset G$ a subset. In additive combinatorics (\cite{Gr13a}), a {\it sequence} (over $G_0$) means a finite sequence of terms from $G_0$ where repetition is allowed and the order of the elements is disregarded, and (as usual) we consider sequences as elements of the free abelian monoid with basis $G_0$. Let
\[
S = g_1 \cdot \ldots \cdot g_{\ell} = \prod_{g \in G_0} g^{\mathsf v_g (S)} \in \mathcal F (G_0)
\]
be a sequence over $G_0$. Then $\supp (S) = \{g_1, \ldots, g_{\ell}\} \subset G_0$ is the support of $S$,   $\sigma (S) = g_1+ \ldots + g_{\ell} \in G$ is the sum of $S$,  $|S|=\ell \in \N_0$ is the length of $S$, and  $-S = (-g_1) \cdot \ldots \cdot (-g_{\ell})$. We denote by
\[
\mathcal B (G_0) = \{ S \in \mathcal F (G_0) \mid \sigma (S)=0\}
\]
the monoid of zero-sum sequences over $G_0$. Since the embedding $\mathcal B (G_0) \hookrightarrow \mathcal F (G_0)$ is a divisor homomorphism, $\mathcal B (G_0)$ is a Krull monoid by Property (c). We follow the convention to write $*(G_0)$ instead of $*(\mathcal  B (G_0))$ for all arithmetical concepts $*(H)$ defined for a monoid $H$. In particular, we have $\mathcal R (G_0)=\mathcal R (\mathcal B (G_0))$, $\Ca (G_0) = \Ca (\mathcal B (G_0))$, and so on. If $G_0$ is finite, then the set of atoms $\mathcal A (G_0) = \mathcal A ( \mathcal B (G_0))$ is finite, and
\[
\mathsf D (G_0) = \sup \{ |A| \mid A \in \mathcal A (G_0) \} \in \N
\]
is the {\it Davenport constant} of $G_0$. Let $G$ be a finite abelian group, say $G \cong C_{n_1} \oplus \ldots \oplus C_{n_r}$, where $r = \mathsf r (G) \in \N_0$ is the rank of $G$, $n_1, \ldots, n_r \in \N$ with $1 < n_1 \t \ldots \t n_r$. If $\mathsf D^* (G) = 1 + \sum_{i=1}^r (n_i-1)$, then
\begin{equation} \label{inequality}
\mathsf D^* (G) \le \mathsf D (G) \le |G| \,.
\end{equation}
We have equality on the left hand side for $p$-groups, for groups with rank $\mathsf r (G) \le 2$, and others (see \cite[Corollary 4.2.13]{Ge09a} for groups close to $p$-groups, and \cite{Sc11b, Bh-SP07a} for groups of rank three).

The next lemma reveals the universal role of monoids of zero-sum sequences in the study of general Krull monoids.

\medskip
\begin{lemma} \label{2.2}
Let $H$ be a reduced Krull monoid, $F=\mathcal F (P)$ a free abelian monoid such that the embedding $H \hookrightarrow F$ is a  cofinal divisor homomorphism, and $G=\mathsf q (F)/\mathsf q (H)$. Let $G_P =\{[p] \mid p \in P\} \subset G$ denote the set of classes containing prime divisors, $\widetilde{\boldsymbol \beta} \colon F \to \mathcal F (G_P)$ be the unique homomorphism satisfying $\widetilde \beta (p) = [p]$ for all $p \in P$, and let $\boldsymbol \beta = \widetilde{\boldsymbol \beta} \mid H  \colon H \to \mathcal B (G_P)$.
\begin{enumerate}
\item $\mathcal L (H)= \mathcal L (G_P)$. In particular, $\daleth^* (H) = \daleth^* (G_P)$.

\smallskip
\item $\mathsf c (G_P) \le \mathsf c (H) \le \max \{\mathsf c (G_P),2\} \le \mathsf D (G_P)$.

\smallskip
\item $\mathcal R (G_P) \subset \mathcal R (H)$.
\end{enumerate}
\end{lemma}

\begin{proof}
The assertions on $\mathcal L (H)$ and $\mathsf c (H)$ follow from \cite[Theorem 3.4.10]{Ge-HK06a}. Since $\mathcal L (H)=\mathcal L (G_P)$, we infer that
\[
\daleth^* (H) = \{ \min (L \setminus \{2\}) \mid 2 \in L \in \mathcal L (H)\} = \{ \min (L \setminus \{2\}) \mid 2 \in L \in \mathcal L (G_P)\} = \daleth^* (G_P) \,.
\]
To verify the assertion on $\mathcal R (H)$, let $d \in \mathcal R (G_P)$ be given. Then there is $A \in \mathcal B (G_P)$ and factorizations $Z=U_1 \cdot \ldots \cdot U_k$, $Z'=V_1 \cdot \ldots \cdot V_{\ell} \in \mathsf Z (A)$, where $U_1, \ldots, U_k,V_1, \ldots, V_{\ell} \in \mathcal A (G_P)$, such that $\mathsf d (Z,Z')=d$ and such that there is no $(d-1)$-chain of factorizations of $A$ concatenating $Z$ and $Z'$. We set $\supp (A)=\{g_1, \ldots, g_m\}$ and we choose a prime divisor $p_i \in g_i \cap P$ for each $i \in [1,m]$. Then there is $a\in H$ and $u_1, \ldots, u_k, v_1, \ldots, v_{\ell} \in \mathcal A (H)$ such that $\boldsymbol \beta (a)=A$, $\boldsymbol \beta (u_i)=U_i$, and $\boldsymbol \beta (v_j)=V_j$ for each $i \in [1,k]$ and each $j \in [1, \ell]$ (recall that for every element $c \in H$ we have that $c$ is an atom of $H$ if and only if $\boldsymbol \beta (c)$ is an atom of $\mathcal B (G_P)$ by \cite[Proposition 3.2.3]{Ge-HK06a}).
Then $z=u_1 \cdot \ldots \cdot u_k, z'=v_1 \cdot \ldots \cdot v_{\ell} \in \mathsf Z (a)$ with $\mathsf d (z, z') = d$. Since $\boldsymbol \beta$ is surjective and there is no $(d-1)$-chain of factorizations of $A$ concatenating $Z$ and $Z'$,  there is no $(d-1)$-chain of factorizations of $a$ concatenating $z$ and $z'$. This implies that $d \in \mathcal R (H)$.
\end{proof}

\medskip
As already mentioned in the introduction, the four sets $\daleth^* (H)$, $\Delta (H)$, $\Ca (H)$, and $\mathcal R (H)$, are often pairwise significantly different. In general,  only the (trivial) relations, as gathered in Lemma \ref{2.1}, hold true and all sets are far from being intervals. Indeed, in Proposition \ref{3.2} we will see that
the sets $\daleth^* (H)$, $\mathcal R (H)$, and $\Ca (H)$ may  coincide simultaneously with any given finite subset of $\N_{\ge 2}$. There  is an abundance of examples in the literature demonstrating the diverging behavior of these sets. Thus we just provide three basic examples, and since   numerical monoids have been studied in detail in \cite{B-C-K-R06, C-K-L-N-Z14, Co-Ka16a, N-P-T-W16a}, we restrict the examples in this note to the setting of Krull monoids.

\medskip
\begin{examples} \label{2.3}~

1. By Lemma \ref{2.1}.3, we have $\daleth^* (H) \subset \mathcal R (H)$. Here we provide  an example of a Krull monoid $H$ for which $\daleth^* (H)=\emptyset$ but $\mathcal R (H)\ne \emptyset$. Let $r, n \in \N$ with $r\ge 2$ and $n\ge 3$, $(e_1, \ldots, e_r) \in G^r$  be independent  with $\ord (e_1)=\ldots = \ord (e_r)=n$, $e_0=-(e_1+ \ldots + e_r)$, and $G_0 = \{e_0, \ldots, e_r\}$. We set $H = \mathcal B (G_0)$ and observe that $\mathcal A (G_0) = \{W, U_0, \ldots, U_r\}$ where  $W = e_0 \cdot \ldots \cdot e_r$ and $U_i= e_i^n$ for all $i \in [0,r]$.  Clearly, we have  $W^n = U_0 \cdot \ldots \cdot U_r$, and it is easy to check that $\daleth^* (G_0) = \emptyset$ but $\Ca (G_0) \ne \emptyset$. A detailed discussion of the arithmetic of this monoid can be found in \cite[Proposition 4.1.2]{Ge-HK06a}. In particular,  $\Delta (G_0) = \emptyset $ if and only if $n = r+1$.

\smallskip
2. By  \eqref{inclusion} we have $\daleth^* (H) \subset 2+\Delta (H)$ and by Lemma \ref{2.1} we have $\daleth^* (H) \cup \Ca (H) \subset \mathcal R (H)$. An example of a Krull monoid $H$ with  $\Delta (H)=\emptyset$ (and hence $\daleth^* (H)=\emptyset$) but for which $\Ca (H)$ and hence $\mathcal R (H)$ are both  infinite is discussed in \cite[Example 4.8.11]{Ge-HK06a}.

\smallskip
3. We provide an example of a reduced Krull monoid $H$ with $\Ca (H) \subsetneq \mathcal R (H)$. We proceed in two steps.
To begin with, let $G$ be an additive abelian group, $e_1, e_2 \in G$ be two independent elements with $\ord (e_1)=\ord (e_2)=4$, and $G_0 = \{ e_1, e_2, -e_1-e_2\}$. Then $H_0 = \mathcal B (G_0)$ is a reduced Krull monoid, and clearly we have $\mathcal A (H_0) = \{u_1, u_2, u_3, u_4\}$, where
\[
u_1=e_1^4, u_2=e_2^4, u_3 = (-e_1-e_2)^4, \quad \text{and} \quad u_4 = e_1e_2(-e_1-e_2) \,.
\]
Then $u_4^4 = u_1u_2u_3$ and $\Ca (H_0)=\mathcal R (H_0 ) = \{4\}$.

Now we use the construction presented in \cite[Lemma 4.8.1]{Ge-HK06a} (we use the notation of that lemma, with $d=3$ and $n=4$). Let $\Gamma = \langle v \rangle$ be an infinite cyclic group, $w=u_1u_2u_3v^{-1} \in \mathsf q (H_0) \times \langle v \rangle$, and
\[
H = [u_1,u_2,u_3, u_4, v, w ] \subset \mathsf q (H_0) \times \langle v \rangle \,.
\]
Then, by \cite[Lemma 4.8.1]{Ge-HK06a}, $H$ is a reduced Krull monoid, $\mathcal A (H) = \{u_1,u_2,u_3,u_4, v, w \}$, and, for all $k_1, \ldots, k_4, k, \ell \in \N_0$ we have $u_1^{k_1} \cdot \ldots \cdot u_4^{k_4}v^k w^{\ell} \in H_0$ if and only if $k=\ell$. Obviously, we have $vw=u_1u_2u_3=u_4^4$, and it follows that $\Ca (H) = \{4\}$ and $\mathcal R (H)=\{3,4\}$.
\end{examples}

\medskip
\section{Main results} \label{3}
\medskip

In this section we establish a realization theorem for finite nonempty subsets of $\N_{\ge 2}$ as sets of catenary degrees  (answering Problem 4.1 in \cite{N-P-T-W16a} in the affirmative), and we give the proof of our main result stated in the Introduction.

\medskip
\begin{lemma} \label{3.1}
Let $n \in \N$, $(H_i)_{i=1}^n$ be a family of \BF-monoids, and $H=H_1 \time \ldots \time H_n$.
\begin{enumerate}
\item $\mathcal R (H) = \bigcup_{i=1}^n \mathcal R (H_i)$.

\smallskip
\item $\Ca (H) = \bigcup_{i=1}^n \Ca (H_i)$.

\smallskip
\item $\daleth^* (H) = \bigcup_{i=1}^n \daleth^* (H_i)$.
\end{enumerate}
\end{lemma}

\begin{proof}
Without restriction we may suppose that $H_1, \ldots, H_n$ are reduced. Then $H_1, \ldots, H_n$ are divisor-closed submonoids of $H$ whence
\[
\bigcup_{i=1}^n \mathcal R (H_i) \subset \mathcal R (H), \quad \bigcup_{i=1}^n \Ca (H_i) \subset \Ca (H), \quad \text{and} \quad \bigcup_{i=1}^n \daleth^* (H_i) \subset \daleth^* (H) \,.
\]
It remains to verify the reverse inclusions. We may suppose that $n=2$. Then the general case follows by an inductive argument. We use some simple properties of the distance function $\mathsf d (\cdot, \cdot)$ in products (details can be found in \cite[Proposition 1.6.8]{Ge-HK06a}).

\smallskip
1. Let $d \in \mathcal R (H)$. Then there are $a=a_1a_2 \in H_1\time H_2$ and factorizations $z=z_1z_2 , z'=z_1'z_2' \in \mathsf Z (a)$ with $\mathsf d (z_1z_2, z_1'z_2')=d$ such that there is no $(d-1)$-chain of factorizations concatenating $z$ and $z'$, where $a_i \in H_i$ and $z_i,z_i' \in \mathsf Z (a_i)$ for $i \in [1,2]$. Then $\mathsf d (z_1z_2, z_1'z_2') = \mathsf d (z_1, z_1')+\mathsf d (z_2, z_2')$, say $d_i = \mathsf d (z_i, z_i')$ for $i \in [1,2]$. Clearly,  $z_1z_2, z_1z_2', z_1'z_2'$ is a $\max \{d_1,d_2\}$-chain from $z$ to $z'$. Thus $\max \{d_1,d_2\}=d_1+d_2$ which implies that $d_1=0$ or $d_2=0$, say $d_2=0$. Thus $z_2=z_2'$ and  $d=d_1 \in \mathcal R (H_1)$.

\smallskip
2. Let $a = a_1a_2 \in H$ with $\mathsf c (a) \ge 2$, where $a_1\in H_1$ and $a_2 \in H_2$. Since $\mathsf c (a) = \max \{\mathsf c (a_1), \mathsf c (a_2) \}$, say $\mathsf c (a) = \mathsf c (a_1)$, it follows that $\mathsf c (a) = \mathsf c (a_1) \in \Ca (H_1)$.

\smallskip
3. Let $d \in \daleth^* (H) \subset \N_{\ge 3}$. Then there are atoms $u,v \in \mathcal A (H)$ such that $\min (\mathsf L (uv) \setminus \{2\})=d$. Note that  $\mathcal A (H) = \mathcal A (H_1) \cup \mathcal A (H_2)$. If $u$ is an atom of $H_1$ and $v$ is an atom of $H_2$, or conversely, then $\mathsf L (uv) = \mathsf L (u)+\mathsf L (v) = \{2\}$. Thus there is an $i \in [1,2]$ such that $u,v$ are atoms of $H_i$ and $d = \min (\mathsf L_H (uv) \setminus \{2\}) = \min (\mathsf L_{H_i} (uv) \setminus \{2\}) \in \daleth^* (H_i)$.
\end{proof}

\medskip
\begin{proposition} \label{3.2}
For every finite nonempty subset $C \subset \N_{\ge 2}$ there is a finitely generated Krull monoid $H$ with finite class group such that $\mathcal R (H) = \Ca (H) = C$, and $\daleth^* (H) = C \setminus \{2\}$.
\end{proposition}

\begin{proof}
Let $s \in \N$ and $C = \{d_1, \ldots, d_s\} \subset \N_{\ge 2}$ be given. We start with the following assertion.

\smallskip
\begin{enumerate}
\item[{\bf A.}\,] For every $i \in [1,s]$, there is a finitely generated reduced Krull monoid $H_i$ with finite class group such that $\mathcal R (H_i) = \Ca (H_i) = \{d_i\}$, and if $d_i>2$, then $\daleth^* (H_i)=\{d_i\}$.
\end{enumerate}

\noindent
{\it Proof of} \,{\bf A}.\,
Let $i \in [1,s]$. First we handle the case where $d_i \ge 3$. Let $G_i$ be a cyclic group of order $d_i$, and  $g_i \in G_i$ an element with $\ord (g_i)=d_i$. We set $H_i = \mathcal B ( \{g_i, -g_i\})$, $U_i = g_i^{d_i}$,  $V_i = (-g_i)g_i$, and  $A_i = U_i(-U_i)$. Then $\mathsf Z (A_i) = \{U_i(-U_i),  V_i^{d_i}\}$ which implies that  $\mathsf c (A_i) = d_i$. Since $\mathcal A (H_i) = \{U_i, -U_i, V_i\}$, we infer that  $\daleth^* (H_i) = \mathcal R (H_i)=\Ca (H_i) = \{d_i\}$.  Clearly, $H_i$ is reduced and finitely generated. Since $\{g_i, -g_i\}$ is a generating set of $G_i$, $H_i$ is a Krull monoid with class group isomorphic to $G_i$ by \cite[Proposition 2.5.6]{Ge-HK06a}.

Now suppose that $d_i=2$. Let $H_i$ be any finitely generated Krull monoid whose class group $G_i$ has exactly two elements. Then $H_i$ is not factorial and $\mathsf c (H_i)=2$. Therefore Lemma \ref{2.1} implies that $\Ca (H_i)=\mathcal R (H_i) = \{2\}$.
\qed[Proof of {\bf A}]

Now $H = H_1 \time \ldots \time H_s$ is a finitely generated Krull monoid with finite class group,  and Lemma \ref{3.1} implies that $\mathcal R (H) = \Ca (H) = \{d_1, \ldots, d_s\} = C$, and that $\daleth^* (H) = C \setminus \{2\}$.
\end{proof}

\medskip
Let $H = H_1 \time \ldots \time H_n$ be a finite direct product of BF-monoids. Then we  have $\Delta (H_1) \cup \ldots \cup \Delta (H_n) \subset \Delta (H)$, but in general this inclusion is strict. For every  finite nonempty set $L \subset \N_{\ge 2}$ there is a finitely generated Krull monoid $H$ and an element $a \in H$ such that $L = \mathsf L (a)$ (\cite[Proposition 4.8.3]{Ge-HK06a}). This implies  that for every finite set $C \subset \N$ there is a finitely generated Krull monoid $H$ and an element $a \in H$ such that $\Delta ( \mathsf L (a)) = C$. However, it is an open problem whether every finite set $C \subset \N$ with $\min C = \gcd C$ (recall Equation \eqref{minequalsgcd}) can be realized as a set of distances of some monoid.
For recent progress in this direction we refer to \cite{Co-Ka16a}. On the other hand, if $H$ is a Krull monoid and every class contains a prime divisor, then $\Delta (H)$ is an interval (\cite{Ge-Yu12b}), and the next proposition reveals that the same holds true for $\daleth^* (H)$.

\medskip
\begin{proposition} \label{3.3}
Let $G$ be a finite abelian group. Then $\daleth^* (G) = \emptyset$ if and only if $|G|\le 2$.  If $|G|\ge 3$, then $\daleth^* (G)$ is a finite interval with $\min \daleth^* (G) = 3$.
\end{proposition}

\begin{proof}
Since $\mathcal B (G)$ is half-factorial if and only if $|G|\le 2$ (\cite[Corollary 3.4.12]{Ge-HK06a}), it follows that $\daleth^* (G) = \emptyset$ if and only if $|G|\le 2$. Suppose that $|G|\ge 3$. Since $\mathsf c (G) \le \mathsf D (G) \le |G|$ by Lemma \ref{2.2} and $\sup \daleth^* (G) \le \sup \mathcal R (G) = \mathsf c (G)$ by Lemma \ref{2.1}, it follows that $\daleth^* (G)$ is finite.
We define a function $f \colon \daleth^* (G) \to \mathbb N$ as follows. If $\ell \in \daleth^* (G)$, then there are  $U_1,  U_2$, $V_1, \ldots, V_{\ell} \in \mathcal A (G)$ such that
\[
U_1 U_2  = V_1 \cdot \ldots \cdot V_{\ell} \,,
\]
where  $U_1 U_2$ has no factorization with length in $[3, \ell-1]$. Let $f (\ell)$ be defined as the minimum over all $|U_1  U_2|$ where $U_1,  U_2$ stem from such a configuration. We continue with the following assertion.

\begin{enumerate}
\item[{\bf A.}\,] For every $\ell \in [3, \max \daleth^* (G)]$, the interval $[\ell, \max \daleth^* (G)]$ is contained in $\daleth^* (G)$ and the function $f \mid [\ell, \max \daleth^* (G)] \colon$ $ [\ell, \max \daleth^* (G)] \to \mathbb N$ is strictly increasing.
\end{enumerate}

\smallskip
\noindent
Clearly, {\bf A} implies that $\daleth^* (G) = [3, \max \daleth^* (G)]$ is an interval with $\min \daleth^* (G) = 3$.

\smallskip
{\it Proof of} \,{\bf A}.\, We proceed by induction on $\ell$. Obviously, the assertion holds for $\ell = \max \daleth^* (G)$. Now suppose that the assertion holds for some $\ell \in [4, \max \daleth^* (G)]$. In order to show that it holds for $\ell-1$,  let $U_1,  U_2, V_1, \ldots, V_{\ell} \in \mathcal A (G)$ be such that
\[
U_1  U_2 = V_1 \cdot \ldots \cdot V_{\ell} \,,
\]
where  $U_1  U_2$ has no factorization with length in $[3, \ell-1]$, and $f(\ell) = |U_1  U_2|$. For every $i \in [1,\ell]$, we set $V_i = U_{1,i}   U_{2,i}$ with $U_{1,i},  U_{2,i} \in \mathcal F (G)$ such that
\[
U_1 = U_{1,1} \cdot \ldots \cdot U_{1,\ell} \quad \text{and} \quad U_2 = U_{2,1} \cdot \ldots \cdot U_{2,\ell} \,.
\]
Note that $f (\ell) = |U_1  U_2| = |V_1 \cdot \ldots \cdot V_{\ell}| \ge 2 \ell \ge 8$. Thus, after renumbering if necessary, we may suppose that $|U_1| \ge 4$, $|U_{1,1}| \ge 1$, and $|U_{1,2}| \ge 1$,  say $g_1 \t U_{1,1}$ and $g_2 \t U_{1,2}$ with $g_1, g_2 \in G$. Clearly,
\[
U_1' = (g_1+g_2)g_1^{-1}g_2^{-1}U_1 \quad \text{and} \quad V_1' = (g_1+g_2)g_1^{-1}g_2^{-1}V_1V_2
\]
are zero-sum sequences, $U_1' \in \mathcal A (G)$, and
\[
U_1'U_2  = V_1'V_3 \cdot \ldots \cdot V_{\ell} \,.
\]

First, we assert that $U_1'U_2$ has no factorization with length in $[3,\ell-2]$. Assume to the contrary that
\[
U_1'U_2  = W_1 \cdot \ldots \cdot W_t \,,
\]
where $t \in [3, \ell-2]$, $W_1, \ldots, W_t \in \mathcal A (G)$ and $(g_1+g_2) \t W_1$. Since $(g_1+g_2)^{-1}g_1g_2W_1$ is either an atom or a product of two atoms, $U_1  U_2$ would have a factorization with length in $[3,\ell-1]$, a contradiction.

Now we assume to the contrary that $U_1' U_2 $ has no factorization with length $\ell-1$. Then there exists a $t \ge \ell$ such that $U_1' U_2 $ has a factorization with length $t$, but no factorization with length in $[3, t-1]$. Therefore, $t\in \daleth^* (G)$,  and
\[
f(t) \le |U_1' U_2| = |U_1  U_2|-1 = f (\ell)-1 \,,
\]
a contradiction to the induction hypothesis that $f$ is strictly increasing on $[\ell, \max \daleth^* (G)]$.
Thus $U_1' U_2$ has a factorization with length $\ell-1$ which implies that $\ell-1 \in \daleth^* (G)$ and $f(\ell-1) \le |U_1' U_2| < |U_1  U_2| = f (\ell)$.
\end{proof}

\medskip
\begin{proposition} \label{3.4}
Let $H$ be a Krull monoid with finite nontrivial class group such that every class contains a prime divisor. If \ $\mathsf D (G)=3$ and each nonzero class contains precisely one prime element, then $\mathcal R (H)=\Ca (H) = \{3\}$. In all other cases we have $\min \Ca (H) = \min \mathcal R (H) = 2$.
\end{proposition}

\begin{proof}
Without restriction we may suppose that $H$ is reduced, and we consider a divisor theory $H \hookrightarrow F = \mathcal F (P)$. The Inequality \eqref{inequality} and the subsequent remark show that  $\mathsf D (G)=3$ if and only if $G$ is cyclic of order three or an elementary $2$-group of rank two. Since $G$ is nontrivial, $H$ is not factorial, and Lemma \ref{2.1} implies that $\Ca (H) \subset \mathcal R (H)$ and that $2 \le \min Ca (H)$.
We distinguish several cases.

\smallskip
\noindent
CASE 1: \, $G$ is an elementary $2$-group.

Suppose that $\mathsf r (G)=1$ (equivalently, $|G|=2$). Since $H$ is not factorial, the nonzero class contains two distinct prime divisors. Thus it is sufficient to consider the following three cases.

\smallskip
\noindent
CASE 1.1: \, There is a nonzero class containing two distinct prime divisors.

Let $g \in G\setminus \{0\}$  contain two distinct prime divisors $p, q \in P \cap g$. Then $u_1=p^2$, $u_2=q^2$, and $v=pq$ are pairwise distinct atoms, and we set $a=u_1u_2$. Then $\mathsf Z (a) = \{u_1u_2, v^2\}$ and hence  $2 = \mathsf c (a) \in \Ca (H)$.

\smallskip
\noindent
CASE 1.2: \, $\mathsf r (G)\ge 3$.

Let $(e_1,e_2,e_3) \in G^3$ be independent and set $e_0=e_1+e_2+e_3$. We choose prime divisors $p_i \in e_i\cap P$ for $i\in [0,3]$, $q_1\in (e_1+e_2)\cap P$, $q_2\in (e_1+e_3)\cap P$, and $q_3 \in (e_2+e_3) \cap P$. Then $u_1=p_0 \cdot \ldots \cdot p_3$, $u_2=q_1q_2q_3$, $v_1=p_1p_2q_1$, and $v_2=p_0p_3q_2q_3$ are pairwise distinct atoms, and we set $a=u_1u_2=v_1v_2$. Then $|\mathsf Z (a)|>1$ and   $2 \le  \mathsf c (a) \le \max \mathsf L (a)=2$ whence $2 = \mathsf c (a) \in \Ca (H)$.

\smallskip
\noindent
CASE 1.3: \, $\mathsf r (G)=2$, and every nonzero class contains precisely one prime divisor.

Suppose that $G = \{0, e_0, e_1,e_2\}$ and let $p_i \in e_i \cap P$ for $i \in [0,2]$. Then $\{p_0^2, p_1^2, p_2^2, p_0p_1p_2\}$ is the set of  atoms which are not prime , and hence $\Ca (H)=\mathcal R (H)=\{3\}$.

\smallskip
\noindent
CASE 2: \, $G$ is an elementary $3$-group.

We distinguish three cases.

\smallskip
\noindent
CASE 2.1: \, There is a nonzero class containing two distinct prime divisors.

Let $g \in G \setminus \{0\}$ contain  two distinct prime divisors  $p,q \in g \cap P$. Then $u_1=p^3, u_2=q^3, v_1=p^2q$, and $v_2=pq^2$ are pairwise distinct atoms, and we set $a=u_1u_2$. Then $\mathsf Z (a) = \{u_1u_2, v_1v_2\}$ and hence  $2 = \mathsf c (a) \in \Ca (H)$.

\smallskip
\noindent
CASE 2.2: \, $\mathsf r (G) \ge 2$.

Let $(e_1, e_2) \in G^2$ be independent and set $e_0=e_1+e_2$. We choose prime divisors $p_2 \in e_2 \cap P$, $p_i' \in (2e_i) \cap P$ for $i \in [1,2]$, and $q \in e_0\cap P$. Then $u_1= q p_1'p_2'$, $u_2=q p_1'p_2p_2$, $v_1=p_2p_2'$, and $v_2=q^2 p_1'p_1' p_2$ are pairwise distinct atoms, and we set $a=u_1u_2$. Then $\mathsf Z (a) = \{u_1u_2, v_1v_2\}$ and hence  $2 = \mathsf c (a) \in \Ca (H)$.

\smallskip
\noindent
CASE 2.3: \, $\mathsf r (G)=1$, and every nonzero class contains precisely one prime divisor.

Suppose that $G=\{0,g, -g\}$, $g \cap P = \{p\}$, and $(-g)\cap P = \{q\}$. Then $\{p^3, q^3, pq\}$ is the set of atoms of $H$ which are not prime, and hence $\Ca (H)=\mathcal R (H)=\{3\}$.

\smallskip
\noindent
CASE 3: \, There is an element $g \in G$ with $\ord (g)=2m+1$ for some $m \ge 2$.

We choose prime divisors $p \in g \cap P$ and $q \in (2g)\cap P$. Then $u_1 = p^{2m+1}, u_2= pq^{m}, v_1=p^{2m-1}q$, and $v_2=q^{m-1}p^3$ are pairwise distinct atoms, and we set $a=u_1u_2=v_1v_2$. Then $|\mathsf Z (a)|>1$ and   $2 \le  \mathsf c (a) \le \max \mathsf L (a)=2$ whence $2 = \mathsf c (a) \in \Ca (H)$.

\smallskip
\noindent
CASE 4: \, There is an element $g \in G$ with $\ord (g)=2m$ for some $m \ge 2$.

We choose prime divisors $p \in g \cap P$ and $q \in (2g)\cap P$. Then $u_1 = p^{2m}, u_2= q^{m}, v_1=p^{2m-2}q$, and $v_2=q^{m-1}p^2$ are pairwise distinct atoms, and we set $a=u_1u_2=v_1v_2$. Then $|\mathsf Z (a)|>1$ and   $2 \le  \mathsf c (a) \le \max \mathsf L (a)=2$  whence  $2 = \mathsf c (a) \in \Ca (H)$.
\end{proof}

\noindent
\begin{proof}[Proof of Theorem \ref{1.1}]
Let $H$ be a Krull monoid with class group $G$ such  that every class contains a prime divisor. Then
Lemma \ref{2.2} implies that $\daleth^* (H)=\daleth^* (G)$, and by Lemma \ref{2.1} we have $\daleth^* (H)\subset \mathcal R (H)  \subset \N_{\ge 2}$.
If $G$ is finite, then $\daleth^* (G)$ is a finite interval by Proposition \ref{3.3}. If $G$ is infinite, then $\daleth^* (H) = \N_{\ge 2}$ by \cite[Theorem 7.4.1]{Ge-HK06a} and hence $\mathcal R (H) = \N_{\ge 2}$.
The details given in items 1 -- 3 of Theorem \ref{1.1} follow from Lemma \ref{2.1} and from the Propositions \ref{3.3}, and \ref{3.4}. Note that, if $|G|=2$, then $\mathsf c (H)=2$ by \cite[Corollary 3.4.12]{Ge-HK06a}.

It remains to prove the equalities given in \eqref{mainresult}.
Thus suppose that $G$ is finite and $\mathsf D (G)=\mathsf D^*(G)\ge 4$.
Since $|G|\ge 3$, $\mathcal B (G)$ is a Krull monoid with  class group isomorphic to $G$ and every class contains a prime divisor (\cite[Proposition 2.5.6]{Ge-HK06a}). In particular, $\mathcal B (G)$ is not factorial whence $\mathsf c (G)\ge 2$ and $\mathsf c (H) = \mathsf c (G) $ by Lemma \ref{2.2}. Furthermore, again by Lemma \ref{2.2} and by Proposition \ref{3.4} (applied to $\mathcal B (G)$), we have
\[
2 \in  \mathcal R (G) \subset \mathcal R (H) \,.
\]
Putting all together we obtain  that
\[
\daleth^* (H)\cup \{2\} =\daleth^* (G) \cup \{2\} \subset \mathcal R (G) \subset \mathcal R (H) \subset [2, \mathsf c (H)]=[2, \mathsf c (G)] \,,
\]
and using Equations \eqref{inclusion} and \eqref{inclusion2} that
\[
\daleth^* (H) \subset 2 + \Delta (H) \subset [2, \mathsf c (H)] \,.
\]
Finally, since $\mathsf D (G)=\mathsf D^*(G)\ge 4$,  \cite[Corollary 4.1]{Ge-Gr-Sc11a} implies that $\max \daleth^* (H) = \mathsf c (H)$ whence equality holds in the above inclusions and the equalities given in  \eqref{mainresult} follow.
\end{proof}

\begin{remarks} \label{3.5}~

1. Let $G$ be a finite nontrivial abelian group,  say $G \cong C_{n_1} \oplus \ldots \oplus C_{n_r}$, where $r, n_1, \ldots, n_r \in \N$ with $1 < n_1 \t \ldots \t n_r$. The equation $\max \daleth^* (H) = \mathsf c (H)$ has not only been proved in case where $\mathsf D^* (G) = \mathsf D (G)$, but under the weaker assumption that
\[
\Big\lfloor\frac{1}{2}\mathsf D (G)+1 \Big\rfloor \le
           \max \Big\{n_r,\,1+\sum_{i=1}^{r} \Big\lfloor\frac{n_i}{2} \Big\rfloor \Big\} \,. \tag{$*$}
\]
There is no known  group where $\max \daleth^* (H) = \mathsf c (H)$ does not hold. The only groups known so far, which satisfy $\mathsf D (G) > \mathsf D^* (G)$ and for which the precise value of $\mathsf D (G)$ is known,  are of the form $G = C_2^4 \oplus C_{2k}$ with  $k \ge 71$ odd (\cite[Theorem 5.8]{Sa-Ch14a}). They satisfy $\mathsf D (G)=\mathsf D^* (G)+1$ whence also the weaker condition in ($*$) holds true.

\smallskip
2. Let $H$ be as in Theorem \ref{1.1} and suppose that $\mathsf D (G)=\mathsf D^* (G)\in \N_{\ge 4}$. The catenary degree $\mathsf c (H)$ is known explicitly only in very special cases (\cite{Ge-Zh15b}) and in all these cases we have $\Ca (H) = \mathcal R (H) = [2,\mathsf c (H)]$. We post the conjecture that this equation holds for all groups $G$ with $\mathsf D (G)=\mathsf D^* (G)\in \N_{\ge 4}$.

\smallskip
3.  The set of elasticities $\{ \rho (L) \mid L \in \mathcal L (H)\}$  has been studied by Chapman et al. in a series of papers (see \cite{B-C-C-K-W06, Ch-Ho-Mo06, B-C-H-M06, Ba-Ne-Pe16a}). They showed that if $H$ is a Krull monoid with finite nontrivial class group $G$, then for every rational number $q$ with $1 \le q \le \mathsf D (G)/2$ there is an $L \in \mathcal L (H)$ with $q = \rho (L)$.

\smallskip
4. All results for Krull monoids dealing with lengths of factorizations carry over to transfer Krull monoids as studied in \cite{Ge16c}. In particular, they hold true for certain maximal orders in central simple algebras over global fields (\cite{Sm13a, Ba-Sm15}).
\end{remarks}

\noindent
{\bf Acknowledgements.} We would like to thank the referee for the careful reading and for all his/her comments.

\providecommand{\bysame}{\leavevmode\hbox to3em{\hrulefill}\thinspace}
\providecommand{\MR}{\relax\ifhmode\unskip\space\fi MR }
\providecommand{\MRhref}[2]{%
  \href{http://www.ams.org/mathscinet-getitem?mr=#1}{#2}
}
\providecommand{\href}[2]{#2}

\end{document}